\documentclass[12pt,a4paper]{article}

\title{\bf Globally linked pairs of vertices in generic frameworks}
\author{
Tibor Jord\'an\thanks{Department of Operations Research, ELTE E\"otv\"os Lor\'and University, and the ELKH-ELTE Egerv\'ary Research Group
on Combinatorial Optimization, E\"otv\"os Lor\'and Research Network (ELKH),
P\'azm\'any P\'eter s\'et\'any 1/C, 1117 Budapest, Hungary.
e-mail: {\tt tibor.jordan@ttk.elte.hu}} 
\and
Soma Vill\'anyi\thanks{Department of Operations Research, ELTE E\"otv\"os Lor\'and University,
P\'azm\'any P\'eter s\'et\'any 1/C, 1117 Budapest, Hungary.
e-mail: {\tt villanyi.soma@gmail.com}}
}
\date{July 7, 2023}

\usepackage[utf8]{inputenc}
\usepackage[english]{babel}
\usepackage{amsmath}
\usepackage{amsthm}
\usepackage{amssymb}
\usepackage{geometry}
\usepackage{a4wide}
\usepackage{bm}
\usepackage{microtype}
\usepackage{mathtools}
\usepackage{csquotes}

\theoremstyle{definition}
\newtheorem{theorem}{Theorem}[section]
\newtheorem{lemma}[theorem]{Lemma}
\newtheorem*{lemma*}{Lemma}
\newtheorem*{conjecture*}{Conjecture}
\newtheorem*{lemma''*}{``Lemma''}

\newtheorem*{claim*}{Claim}
\newtheorem{corollary}[theorem]{Corollary}

\newtheorem{proposition}[theorem]{Proposition}

\newcommand{\R}{\mathbb{R}}

\begin{document}
\maketitle

\begin{abstract}
A $d$-dimensional framework is a pair $(G,p)$, where $G=(V,E)$
is a graph and $p$ is a map from $V$ to $\mathbb{R}^d$.
The length of an edge $xy\in E$ in $(G,p)$ is the distance between $p(x)$ and $p(y)$.
A vertex pair $\{u,v\}$ of $G$ is said to be globally linked in $(G,p)$ if the
distance between $p(u)$ and $p(v)$ is equal to the distance between $q(u)$ and $q(v)$
for every $d$-dimensional framework $(G,q)$ in which the corresponding edge lengths are the same as in $(G,p)$.
We call $(G,p)$ globally rigid in $\R^d$ when each vertex pair of $G$ is globally linked in $(G,p)$.
A pair $\{u,v\}$ of vertices of $G$ is said to be weakly globally linked in $G$ in $\R^d$ if 
there exists a generic framework $(G,p)$ in which $\{u,v\}$ is globally linked.

In this paper we first give a sufficient condition for the weak global linkedness of a vertex pair of a
$(d+1)$-connected graph $G$ in $\R^d$ and then show that for $d=2$ it is also necessary.
We use this result to obtain a complete characterization of weakly globally linked pairs
in graphs in $\R^2$, which gives rise to an algorithm for testing weak global linkedness in the
plane in $O(|V|^2)$
time. Our methods lead to
a new short proof for the characterization of globally rigid graphs in $\R^2$, and further
results on weakly globally linked pairs and globally rigid graphs in the plane and in higher
dimensions.
\end{abstract}

\section{Introduction}

A $d$-dimensional {\it framework}
is a pair $(G,p)$, where $G=(V,E)$
is a graph and $p$ is a map from $V$ to $\mathbb{R}^d$.
We also say that $(G,p)$ is a {\it realization} of $G$ in
$\mathbb{R}^d$. The {\it length} of an edge $uv\in E$ in
$(G,p)$ is $||p(u)-p(v)||$, where $||.||$ denotes the Euclidean norm in $\mathbb{R}^d$.
Two frameworks
$(G,p)$ and $(G,q)$ are {\it equivalent} if corresponding edge lengths are the same, that is, 
$||p(u)-p(v)||=||q(u)-q(v)||$ holds
for all pairs $u,v$ with $uv\in E$.
The frameworks $(G,p)$ and $(G,q)$ are {\it congruent} if
$||p(u)-p(v)||=||q(u)-q(v)||$ holds
for all pairs $u,v$
with $u,v\in V$.

A $d$-dimensional framework $(G,p)$ is called {\it globally rigid} if every equivalent 
$d$-dimensional framework
$(G,q)$ is congruent to $(G,p)$. This is the same as saying that the edge lengths of $(G,p)$
uniquely determine all the pairwise distances.
It is NP-hard to test whether a given framework in $\R^d$ is globally
rigid, even for $d=1$ \cite{Saxe}.
This fundamental property of frameworks becomes more tractable if we consider generic
frameworks. A framework $(G,p)$ (and the set $\{p(v):v\in V(G)\}$) is said to be {\it generic} if the set of
its $d|V(G)|$ vertex coordinates is algebraically independent over $\mathbb{Q}$.
It is known that in a given dimension 
the global rigidity of a generic framework $(G,p)$ depends only on $G$: either every generic realization of $G$ in $\R^d$ is globally rigid, or none of them are
\cite{Con, GHT}.
Thus, we say that a graph $G$ is {\it globally rigid} in $\R^d$ if every (or equivalently, if some) $d$-dimensional
generic realization of $G$ is globally rigid in $\R^d$.
For $d=1,2$, combinatorial characterizations and corresponding
deterministic polynomial time algorithms are known for (testing) 
global rigidity in $\R^d$.
The case $d=1$ is a folklore result:
it is not hard to see that a graph $G$ on at least three vertices is globally rigid in $\R^1$ if and only
if it is $2$-connected. The necessary and sufficient conditions for $d=2$ are stated as Theorem \ref{thm}
in the next section.
The existence of such a characterization
(or algorithm) for $d\geq 3$ is a major open question.
For more details on globally rigid graphs and frameworks see e.g.\ \cite{JW}.

In this paper we consider a refined, local version, in which we are interested in whether the edge lengths
of a framework uniquely determine the distance between a given pair of vertices, rather than all pairs of vertices.
We shall need the following notions. Following \cite{JJS}, 
we say that
a pair of vertices $\{u,v\}$ in a $d$-dimensional framework $(G,p)$ is {\it globally linked in $(G,p)$} if for every equivalent $d$-dimensional framework $(G,q)$ we have
$||p(u)-p(v)||=||q(u)-q(v)||$. Global linkedness in $\R^d$ is not a generic property (for $d\geq 2$): 
a vertex pair may be globally linked in some generic $d$-dimensional realization of $G$ without being globally linked in all generic realizations. 
See Figure \ref{fig:gen1}. 
We say that
a pair $\{u,v\}$ is {\it globally linked in $G$} in $\R^d$ if it is globally linked in all generic $d$-dimensional frameworks $(G,p)$. 
We call a pair $\{u,v\}$ {\it weakly globally linked in $G$} in $\R^d$ if there exists a generic $d$-dimensional framework $(G,p)$
in which $\{u,v\}$ is globally linked. If  $\{u,v\}$ is not weakly globally linked in $G$, then it is called {\it globally loose} in $G$.
It is immediate from the definitions that $G$ is globally rigid in $\R^d$ if and only if each vertex pair is globally linked in $G$ in $\R^d$. As we shall see, the global rigidity of $G$ 
 already follows from the (seemingly weaker) condition that each vertex pair is weakly globally linked in $G$ (see Lemma \ref{J}(c)). 

The case $d=1$ is exceptional and well-understood. Global linkedness in $\R^1$ is a generic property: a pair $\{u,v\}$  is globally linked
in $G$ in $\R^1$ if and only if there is a cycle in $G$ that contains both $u$ and $v$.
Otherwise $\{u,v\}$ is globally loose.

For $d\geq 2$ no combinatorial (or efficiently testable) characterization has previously been found for globally linked or weakly globally linked pairs in graphs in $\R^d$.
These problems belong to the few major problems in combinatorial rigidity which have remained unsolved for $d=2$.
The main result of this paper is a solution for the weakly globally linked pairs problem in two dimensions.
We shall first give a sufficient condition for the weak global linkedness of a vertex pair of a
$(d+1)$-connected graph $G$ in $\R^d$ (Theorem \ref{globlaza jellemzesd})
and then show that in a sense the condition is also necessary in the case of
$3$-connected graphs in $\R^2$ (Theorem \ref{globlaza jellemzes2}).
The general case of the two-dimensional problem
is reduced to the $3$-connected case by a sequence of lemmas that describe how global linkedness is affected by cutting a graph along a separating pair. These results lead to
the main result (Theorem \ref{cleavunit}), which gives a characterization of weakly globally linked
pairs of vertices in $\R^2$ and gives rise to an $O(|V|^2)$ algorithm for the corresponding decision
problem.

Our methods and results lead to
a new short proof for the sufficiency part of Theorem \ref{thm}.
We also obtain a number of other structural results on weakly globally linked pairs and globally rigid graphs in $\R^2$
and in higher dimensions.

Even though
most of the known results (and conjectures) on global linkedness are concerned with
globally linked pairs of graphs in $\R^2$, their characterization remains open.
Globally linked pairs in two dimensions
have been characterized in minimally rigid graphs \cite{JJS2},
braced maximal outerplanar graphs \cite{GJcccg}, 
 and in
${\cal R}_2$-connected graphs \cite{JJS}.
In the latter two cases global linkedness turns out to be a generic property.
Hence these two results
give rise to the characterization of weakly globally linked pairs, too, in the corresponding families
of graphs.
A conjectured characterization of globally linked pairs in $\R^2$ can be found in \cite{JJS}.
A few partial results in higher dimensions are also available, see \cite{GJpartial,JT}.

The rest of the paper is organized as follows.
In Section \ref{sec:pre}
we introduce the necessary notions concerning rigid graphs and frameworks.
In Section \ref{sec:pro} we prove some simple but fundamental lemmas on weakly
globally linked pairs in $\R^d$.
Section \ref{sec:suff} contains most of the $d$-dimensional results (two key geometric lemmas
and
a sufficient condition for weak global linkedness), and the new proof for Theorem \ref{thm}. 
In Section \ref{sec:2d} we state and prove our main result, a complete characterization of the
weakly globally linked pairs in $\R^2$. In Section \ref{section:concluding} we
discuss the algorithmic aspects and collect a few
concluding remarks and
questions.

\begin{figure}[t]
\begin{center}
\includegraphics[scale=1.1]{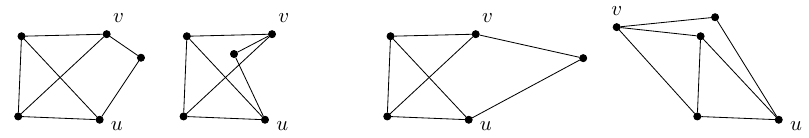}
\vspace{-0.5cm}
\end{center}
\caption{Two pairs of equivalent generic frameworks of a graph $G$ in $\R^2$. The vertex pair $\{u,v\}$ is globally linked in the two frameworks on the left. On the other hand, $\{u,v\}$ is not globally linked in the two frameworks on the right. Thus $\{u,v\}$ is not globally linked but weakly globally linked in $G$ in $\R^2$.}
\label{fig:gen1}
\end{figure}

\section{Preliminaries}
\label{sec:pre}

In this section we introduce the notions and results from the
theory of (globally) rigid frameworks and graphs that we shall use.

\subsection{Rigid graphs and the rigidity matroid}

In the structural results on global rigidity and global linkedness the notions of
rigid frameworks, rigid graphs and the rigidity matroid play a key role.

The $d$-dimensional framework $(G,p)$ is {\it rigid} if there exists some $\varepsilon
>0$ such that, if $(G,q)$ is equivalent to $(G,p)$ and
$||p(v)-q(v)||< \varepsilon$ for all $v\in V$, then $(G,q)$ is
congruent to $(G,p)$. 
This is equivalent to requiring that every continuous motion of the vertices of $(G,p)$ in $\R^d$ that
preserves the edge lengths takes the framework to a congruent realization of $G$.
It is known that in a given dimension 
the rigidity of a generic framework $(G,p)$ depends only on $G$: either every generic realization of $G$ in $\R^d$ is rigid, or none of them are \cite{AR}.
Thus, we say that a graph $G$ is {\it rigid} in $\R^d$ if every (or equivalently, if some) $d$-dimensional
generic realization of $G$ is rigid in $\R^d$.
For $d=1,2$, combinatorial characterizations and corresponding
deterministic polynomial time algorithms are known for (testing) rigidity in $\R^d$, see e.g. \cite{laman}.
The existence of such a characterization
(or algorithm) for $d\geq 3$ is a major open question.

The following elementary result is well-known.
For the proof of the two-dimensional case see 
\cite[Theorem 8.1]{JJS}.

\begin{proposition}\label{finiteequi}
Suppose that $(G,p)$ is a rigid generic framework. Then the number of distinct congruence classes of frameworks which are equivalent to $(G,p)$ is finite.
\end{proposition}

The rigidity matroid of a graph $G$ is a matroid defined on the edge set
of $G$ which reflects the rigidity properties of all generic realizations of
$G$.
For a general introduction to matroid theory we refer the reader to \cite{oxley}. 

Let $(G,p)$ be a realization of a graph $G=(V,E)$ in $\R^d$.
The \emph{ rigidity matrix} of the framework $(G,p)$
is the matrix $R(G,p)$ of size
$|E|\times d|V|$, where, for each edge $uv\in E$, in the row
corresponding to $uv$,
the entries in the $d$ columns corresponding to vertices $u$ and $v$ contain
the $d$ coordinates of
$(p(u)-p(v))$ and $(p(v)-p(u))$, respectively,
and the remaining entries
are zeros. 
The rigidity matrix of $(G,p)$ defines
the \emph{rigidity matroid}  of $(G,p)$ on the ground set $E$
by linear independence of the rows. %
It is known that any pair of generic frameworks
$(G,p)$ and $(G,q)$ have the same rigidity matroid.
We call this the $d$-dimensional \emph{rigidity matroid}
${\cal R}_d(G)=(E,r_d)$ of the graph $G$.

We denote the rank of ${\cal R}_d(G)$ by $r_d(G)$.
A graph $G=(V,E)$ is \emph{${\cal R}_d$-independent} if $r_d(G)=|E|$ and it is an \emph{${\cal R}_d$-circuit} if it is not ${\cal R}_d$-independent but every proper 
subgraph $G'$ of $G$ is ${\cal R}_d$-independent. We note that in the literature such graphs are sometimes called $M$-independent in $\R^d$ and $M$-circuits in $\R^d$, respectively. 
An edge $e$ of $G$ is an \emph{${\cal R}_d$-bridge in $G$}
if  $r_d(G-e)=r_d(G)-1$ holds. Equivalently, $e$ is an ${\cal R}_d$-bridge in $G$ if it is not contained in any subgraph of $G$ that is an ${\cal R}_d$-circuit.

The following characterization of rigid graphs is due to Gluck.
\begin{theorem}\label{theorem:gluck}
\cite{Gluck}
\label{combrigid}
Let $G=(V,E)$ be a graph with $|V|\geq d+1$. Then $G$ is rigid in $\R^d$
if and only if $r_d(G)=d|V|-\binom{d+1}{2}$.
\end{theorem}

A graph is \emph{minimally rigid} in $\R^d$ if it is rigid in $\R^d$ but 
$G-e$ is not rigid in $\R^d$ for
every edge $e$ of $G$. 
By Theorem \ref{theorem:gluck}, minimally rigid graphs in $\R^d$ on at least $d+1$ vertices have exactly $d|V| - \binom{d+1}{2}$ edges.

Let $G=(V,E)$ be a graph and $\{u,v\}$ be a pair of vertices of $G$. 
An induced subgraph $G[X]$ (and the set $X$), for some $X\subseteq V$, is said to be
{\it $(u,v)$-rigid} in $\R^d$ (or simply $(u,v)$-rigid, if $d$ is clear from the context),
if $G[X]$ is rigid in $\R^d$ and $u,v\in X$. 
We say that 
a $(u,v)$-rigid subgraph
$G[X]$ is
{\it vertex-minimally} $(u,v)$-rigid, if $G[X']$ is not $(u,v)$-rigid
for all proper subsets $X'\subset X$.
The pair $\{u,v\}$ is called {\it linked} in $G$ in $\R^d$ if
$r_d(G+uv)=r_d(G)$ holds. It is known that a pair
$\{u,v\}$ is linked in $G$ in $\R^2$ if and only if there exists a $(u,v)$-rigid subgraph of $G$.
A graph $G$ with at least three edges is called
{\it redundantly rigid} in $\R^d$ if $G-e$ is rigid
in $\R^d$
for
all $e\in E(G)$.

Let ${\cal M}$ be a matroid on ground set $E$. 
We can define a relation on the pairs of elements of $E$ by
saying that $e,f\in E$ are
equivalent if $e=f$ or there is a circuit $C$ of ${\cal M}$
with $\{e,f\}\subseteq C$.
This defines an equivalence relation. The equivalence classes are 
the \emph{connected components} of ${\cal M}$.
The matroid is \emph{connected} if 
it has only one connected component.
A graph $G=(V,E)$ is \emph{${\cal R}_d$-connected} if ${\cal R}_d(G)$ is connected.
We shall use the well-known fact that 
if $v$ is a vertex of degree at most $d$ in $G$, then every edge incident with $v$ is an
${\cal R}_d$-bridge in $G$. Hence the addition of a new vertex of degree $d$ to a rigid graph $G$ 
in $\R^d$ preserves rigidity.

For more details on the $2$-dimensional rigidity matroid, see \cite{Jmemoirs}.

\subsection{Globally rigid graphs}

The following necessary conditions for global rigidity are due to Hendrickson.

\begin{theorem} \cite{hend}
    \label{thm:hend}
    Let $G$ be a globally rigid graph in $\R^d$ on at least $d+2$ vertices. Then $G$ is $(d+1)$-connected and
    redundantly rigid in $\R^d$.
\end{theorem}

For $d=1,2$
the conditions of Theorem \ref{thm:hend} together are sufficient to imply global rigidity.
It is not the case for $d\geq 3$.
The characterization of globally rigid graphs in $\R^2$ is as follows.

\begin{theorem} \cite{JJconnrig}
\label{thm} 
Let $G$ be a graph on at least four vertices.
Then $G$ is globally rigid in $\R^2$ if and only if $G$ is $3$-connected and redundantly rigid in $\R^2$.
\end{theorem}

An equivalent characterization of global rigidity, in terms of the rigidity matroid of $G$, follows
from the next lemma.

\begin{lemma} \cite[Lemma 3.1, Theorem 3.2]{JJconnrig}
\label{rrmc}
Let $G$ be a graph with at least two edges.
If $G$ is ${\cal R}_2$-connected, then $G$ is redundantly rigid in $\R^2$.
Furthermore, if $G$ is $3$-connected and redundantly rigid in $\R^2$, then $G$ is 
${\cal R}_2$-connected.
\end{lemma}

We shall also use the following lemma.

\begin{lemma} \cite[Lemma 6.2]{JKT}
    \label{rrtri}
    Let $G$ be a rigid, but not redundantly rigid graph in $\R^2$, and suppose that all ${\cal R}_2$-bridges of $G$
    are edges of the same triangle in $G$. Then $G$ is not 3-connected.
\end{lemma}

\section{Properties of weakly globally linked pairs in $\R^d$}
\label{sec:pro}

We first collect some basic properties that hold in $\R^d$ for all $d\geq 1$.
The following lemma was stated for $d=2$ in \cite{JJS} but the proof works
for all $d\geq 1$.
An edge $e$ of a globally rigid graph $H$ is {\it critical} if $H-e$ is not
globally rigid.

\begin{lemma} \cite[Lemma 7.1]{JJS}
\label{critical}
Let $G=(V,E)$ be a graph and $u,v\in V$. Suppose that $uv\notin E$, and that
$G$ has a globally rigid supergraph in $\R^d$ in which $uv$ is a critical edge.
Then $\{u,v\}$ is globally loose in $G$ in $\R^d$.
\end{lemma}

We shall frequently use the next key lemma. For a graph $G=(V,E)$ and integer $d\geq 1$
let
$$J_d(G)=\{uv : u,v\in V, uv\notin E, \{u,v\} \ \hbox{is weakly globally linked in}\ G\ \hbox{in}\ \R^d\}.$$

\begin{lemma}
\label{J}
Let $G=(V,E)$ be a graph and let $F$ be a set of edges on vertex set $V$.  
Then the following hold.\\ 
(a) If $G+J_d(G)+F$ is globally rigid in $\R^d$, then $G+F$ is globally rigid in $\R^d$.\\
(b) If $G+uv$ is globally rigid in $\R^d$ for some $uv\in J_d(G)$, then $G$ is globally rigid in $\R^d$.\\
(c) $G$ is globally rigid in $\R^d$ if and only if 
all pairs of vertices in $G$ are weakly globally linked in $\R^d$.
\end{lemma}

\begin{proof} Let us fix $d$ and put $J=J_d(G)$.

(a) Suppose, for a contradiction, that $G+J+F$ is globally rigid and $G+F$ is not. Then there is a (possibly empty) subset $J'\subset J$ and an edge $uv\in J-J'$
for which $G+J'+F$ is not globally rigid, but $\bar G=G+J'+F+uv$ is globally rigid. Then $uv$ is a critical
edge
in $\bar G$, and hence $\{u,v\}$ is globally loose in $G$ by Lemma \ref{critical},
a contradiction.

(b) If $G+uv$ is globally rigid for some $uv\in J$ then $G+J$ is globally rigid. Thus putting $F=\emptyset$ and applying (a)
gives that $G$ is globally rigid.

(c)  Necessity is obvious. If all pairs of vertices in $G$ are weakly globally linked, then $G+J$ is a complete graph, which is globally rigid.
Again, putting $F=\emptyset$ and applying (a)
gives that $G$ is globally rigid.
\end{proof}

It is well-known that if  $\{u,v\}$ is not linked in $G$ in $\R^d$, then every generic $d$-dimensional realization $(G,p)$ has a flex (i.e. a continuous motion of the vertices that preserves the edge lengths)
to another framework $(G,q)$, for which $||p(u)-p(v)||\not= ||q(u)-q(v)||$.
This implies the next lemma.

\begin{lemma}
\label{notlinked}
Let $G=(V,E)$ be a graph and let $\{u,v\}$ be a non-adjacent vertex pair.
If $\{u,v\}$ is not linked in $G$ in $\R^d$ then $\{u,v\}$ is globally loose in $G$ in $\R^d$.
\end{lemma}

Let $H=(V,E)$ be a graph and $x,y\in V$. We use $\kappa_H(x,y)$ to
denote the maximum number of pairwise internally disjoint $xy$-paths in
$H$. Note that if $xy\notin E$ then, by Menger's theorem,
$\kappa_H(x,y)$ is equal to the size of a smallest set $S\subseteq
V-\{x,y\}$ for which there is no $xy$-path in $H-S$.
The following lemma is the $d$-dimensional and slightly stronger 
version of \cite[Lemma 5.6]{JJS}.

\begin{lemma}
\label{kappa}
Let $G=(V,E)$ be a graph and let $\{u,v\}$ be a non-adjacent vertex pair
with $\kappa_G(u,v)\leq d$. Then $\{u,v\}$ is globally loose in $G$ in $\R^d$.
\end{lemma}

Let $G_i=(V_i,E_i)$ be a graph, $t \geq 1$ an integer, and suppose that $K_i$ is a complete subgraph of $G_i$
on $t$ vertices, for $i=1,2$. Then the {\it $t$-clique sum} operation on $G_1,G_2$,
along $K_1,K_2$, creates a new graph $G$ by identifying the vertices of $K_1$ with the vertices
of $K_2$, following some bijection between their vertex sets. 
The {\it clique sum} operation is a $t$-clique sum operation for some $t\geq 1$.

In the following lemma sufficiency follows from the simple obervation that if a vertex pair is
weakly globally linked in a subgraph of $G$, then it is also weakly
globally linked in $G$. Necessity follows from the fact that the clique sum
operation is performed along a complete (and hence globally rigid) subgraph.

\begin{lemma}
\label{cliquesum}
Suppose that $G$ is the clique sum of $G_1$ and $G_2$ and let $u,v\in V(G_1)$.
Then $\{u,v\}$ is weakly globally linked in $G$ in $\R^d$ if and only if
$\{u,v\}$ is weakly globally linked in $G_1$ in $\R^d$.
\end{lemma}

\section{A sufficient condition for weak global linkedness in $\R^d$}
\label{sec:suff}

In this section we provide a new sufficient condition 
for the weak global linkedness of a pair of vertices of 
a $(d+1)$-connected graph in $\R^d$.
An important ingredient in our proof is a
geometric lemma (Lemma \ref{lem gllaza1}) presented in the next subsection. In Subsection \ref{sec:suffcond} we prove the aforementioned sufficient condition and in Subsection \ref{sec:newproof} we show how it can be used to prove the sufficiency part of Theorem \ref{thm}. 
In the last subsection 
we shall see that an appropriate reverse of Lemma \ref{lem gllaza1} is also true (Lemma \ref{lem gllaza2}).
This lemma will be used in the next section where we  characterize weak global linkedness in two dimensions.
Roughly speaking these two lemmas show that if a vertex pair $\{u,v\}$ 
belongs to a rigid subgraph $H$ of $G$, then 
the contraction of a
connected subgraph of $G-V(H)$ 
does not change the weak global linkedness properties of $\{u,v\}$.

\subsection{The first contraction lemma}

A basic graph operation is the {\it contraction} of a subset $V_0$ of $V$
in the graph $G=(V,E)$. 
This operation, which is denoted by $G/V_0$, identifies the vertices of $V_0$ and removes
the loops and parallel copies of the edges of the resulting graph that it may create. 
The contraction of an edge $e=xy$ is the contraction of the set $\{x,y\}$ and it is denoted by $G/e$.

\begin{lemma}
\label{lem gllaza1}
Let $G=(V,E)$ be a graph, $u,v\in V$, and suppose that $G[V_0]$ is a $(u,v)$-rigid subgraph of $G$.
Let 
$e=(s_1,s_2)\in E-E(G[V_0])$ be an edge.
If $\{u,v\}$ is weakly globally linked in $G/e$ in $\R^d$, then $\{u,v\}$ is weakly globally linked in $G$ in $\R^d$.
\end{lemma}

\begin{proof}
We may assume that $G$ is connected and $s_2\notin V_0$. 
Let $s$ denote the vertex of $G/e$ obtained by identifying $s_1$ and $s_2$ in $G$. 
Note that we may have $s_1\in V_0$.
In this case we shall
simply identify $s$ with $s_1$
for notational convenience.
Let $(G/e,p)$ be a generic realization of $G/e$ in which 
$\{u,v\}$ is globally linked.
Let $(G,p_i)$ be a sequence of generic realizations of $G$, for which
$p_i|_{V-s_1-s_2}=p|_{V-s}$, $p_i(s_1)=p(s)$, and $p_i(s_2)\to p(s)$. 
Suppose, for a contradiction, that $\{u,v\}$ is globally loose in $G$. Then $\{u,v\}$ is not globally linked in $(G,p_i)$ for all $i\geq 1$. Hence for all $i\geq 1$ there exists
a realization $(G,p_i')$, equivalent to $(G,p_i)$, for which
$$||p_i'(u)-p_i'(v)||\neq ||p_i(u)-p_i(v)||=||p(u)-p(v)||.$$ 
Since
$G[V_0]$ is rigid and $p|_{V_0}=p_i|_{V_0}$, it follows from Proposition \ref{finiteequi} that there is an
$\epsilon >0$ such that for all $i\geq 1$,
$$\big|||p_i'(u)-p_i'(v)||-||p(u)-p(v)||\big|\geq \epsilon.$$

Since $G$ is connected, we can translate each framework, if necessary, so that  for all $i\geq 1$,
$(G,p_i')$ is in the interior of a ball of radius $K$, centered at the origin, for some fixed 
positive real number $K$. Thus 
there is a convergent subsequence $p_{i_k}'\to p'$. Since $(s_1,s_2)\in E$, we must have $p'(s_1)=p'(s_2)$.
By extending $p'|_{V-s_1-s_2}$ with $p'(s)=p'(s_1)$, we obtain a realization $(G/e,p')$ which is equivalent to $(G/e,p)$. Furthermore, we have 
$$\big|||p'(u)-p'(v)||-||p(u)-p(v)||\big|\geq \epsilon,$$ 
which contradicts the fact that $\{u,v\}$ is globally linked in $(G/e,p)$.
Thus $\{u,v\}$ is weakly globally linked in $G$.
\end{proof}

We obtain the following sufficient (but not necessary, see Figure \ref{fig:gen4}) condition for weak global linkedness as a corollary. 

\begin{corollary}
\label{coro}
Let $G=(V,E)$ be a graph, $u,v\in V$. Suppose that there is some $V_0\subset V$ such that $G[V_0]$ is a $(u,v)$-rigid
subgraph of $G$ in $\R^d$, and
there is a $uv$-path in $G$ that is internally disjoint from $V_0$. Then
$\{u,v\}$ is weakly globally linked in $G$ in $\R^d$.
\end{corollary}

\begin{figure}[ht]
\centering
\includegraphics[scale=0.9]{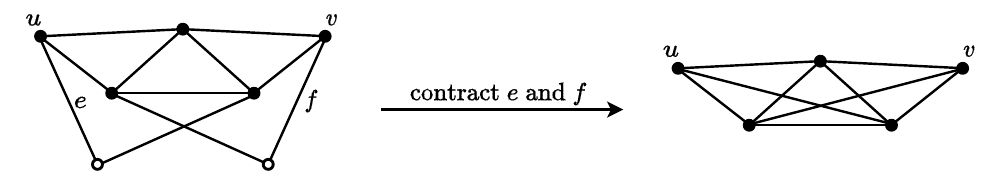}
\caption{Consider the graph on the left.  The subgraph induced by the solid vertices is rigid in $\R^2$, and contracting the edges $e$ and $f$ results in a graph that is globally rigid in $\R^2$. Thus by Lemma \ref{lem gllaza1}, $\{u,v\}$ is weakly globally linked in $\R^2$. This shows that the sufficient condition of Corollary \ref{coro} is not necessary.}
\label{fig:gen4}
\end{figure}

Corollary \ref{coro}, together with Lemma \ref{J}, leads to short proofs for some previous results on
globally rigid graphs. We illustrate this by the following theorem.

\begin{theorem} \cite{Conmerge}
Let $G_1$ and $G_2$ be two globally rigid graphs in $\R^d$ on at least $d+2$ vertices,
with exactly $d+1$ vertices in common. Suppose that $e$ is a common edge. Then $G=G_1\cup G_2-e$
is globally rigid in $\R^d$.
\end{theorem}

\begin{proof}
Let $e=uv$. Theorem \ref{thm:hend} implies that $G_1-e$ is rigid.
Since $G_2$ is $(d+1)$-connected, there is a path from $u$ to $v$ in $G$ that is internally disjoint
from $G_1$. Thus $\{u,v\}$ is weakly globally linked in $G$ by Corollary \ref{coro}.
It is easy to see that $G+uv$ is globally rigid. Hence $G$ is also globally rigid by Lemma \ref{J}.
\end{proof}

By using the same proof idea we obtain a simple proof of the “rooted minor" theorem of Tanigawa \cite{Tani}.

\subsection{The sufficient condition}\label{sec:suffcond}

Let $G=(V,E)$ be a graph, $\emptyset\not= X\subseteq V$, and let $V_1,V_2,\dots, V_r$ be the
vertex sets of the connected components of $G-X$.
The graph ${\rm Con}(G,X)$ is obtained from $G$ by contracting each vertex set $V_i$ into a single vertex $v_i$, $1\leq i\leq r$.
The graph ${\rm Clique}(G,X)$ is obtained from $G$ by deleting the vertex sets $V_i$, $1\leq i\leq r$, and adding
a new edge $xy$ for all pairs $x,y\in N_G(V_i)$, $xy\notin E$, for $1\leq i\leq r$.
See Figure \ref{fig:gen6}.

\begin{figure}[ht]
\centering
\includegraphics[scale=0.95]{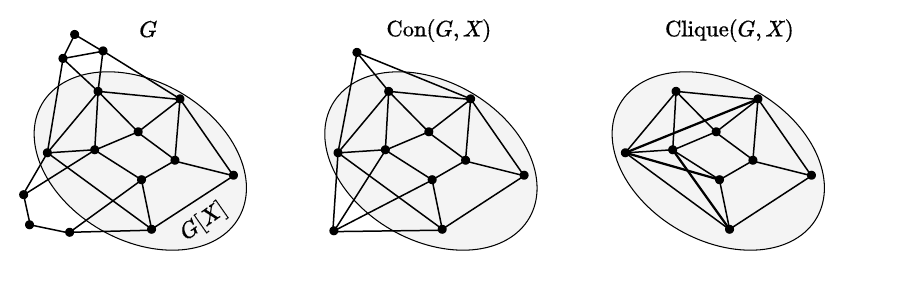}
\vspace{-0.7cm}
\caption{An example of a graph $G$ and the graphs ${\rm Con}(G,X)$ and ${\rm Clique}(G,X)$.}
\label{fig:gen6}
\end{figure}

\begin{lemma}
\label{cceq}
Let  $G=(V,E)$ be a $(d+1)$-connected graph.
Suppose that $G[V_0]$ is a rigid subgraph of $G$ for some $V_0\subseteq V$. 
Then ${\rm Clique}(G,V_0)$ is globally rigid in $\R^d$
if and only if ${\rm Con}(G,V_0)$ is globally rigid in $\R^d$.
\end{lemma}
\begin{proof}
Let $E'$ be the set of those edges in ${\rm Clique}(G,V_0)$ that are not in $G[V_0]$. Let $H={\rm Con}(G,V_0)+E'$. It follows from 
Corollary \ref{coro} that $\{u,v\}$ is weakly globally linked in ${\rm Con}(G,V_0)$ for all $uv\in E'$. Hence, by Lemma \ref{J}, ${\rm Con}(G,V_0)$ is globally rigid if and only if $H$ is globally rigid. $H$ can be obtained from ${\rm Clique}(G,V_0)$ by adding new vertices and joining them to cliques of size at least $d+1$. Thus $H$ is globally rigid if and only if ${\rm Clique}(G,V_0)$ is globally rigid.
\end{proof}

We are ready to state the main result of this section.

\begin{theorem}
\label{globlaza jellemzesd}
Let 
$G=(V,E)$ be a $(d+1)$-connected graph and $u,v\in V$.
Suppose that $G[V_0]$ is a $(u,v)$-rigid subgraph of $G$ in $\R^d$. If ${\rm Clique}(G,V_0)$ is globally rigid in $\R^d$,
then $\{u,v\}$ is weakly globally linked in $G$ in $\R^d$.
\end{theorem}
\begin{proof}
Suppose that ${\rm Clique}(G,V_0)$ is globally rigid in $\R^d$. Then so is ${\rm Con}(G,V_0)$ by Lemma \ref{cceq}. In particular,
$\{u,v\}$ is weakly globally linked in ${\rm Con}(G,V_0)$.
Since ${\rm Con}(G,V_0)$
can be obtained from $G$ by contracting edges not induced by $V_0$, Lemma
\ref{lem gllaza1} gives that $\{u,v\}$ is weakly globally linked in $G$.
\end{proof}

\subsection{Globally rigid graphs - a new proof}\label{sec:newproof}

Theorem \ref{globlaza jellemzesd} and Lemma \ref{J} lead to a new short proof of the sufficiency
part of Theorem \ref{thm}, which only uses 
the simple combinatorial Lemmas \ref{rrmc} and \ref{rrtri} and
 the fact that the global rigidity
of graphs in $\R^2$ is a generic property.
The original proof in \cite{JJconnrig} relies on an inductive construction of $3$-connected
${\cal R}_2$-connected graphs.

\begin{proof} (of sufficiency in {\bf Theorem \ref{thm}})
The proof is by induction on $|V|$. If $|V|=4$ then $G$ is a complete graph on four
vertices, which is globally rigid. So we may suppose that $|V|\geq 5$.
First, we show that for all non-adjacent pairs $u,v$, there is a $(u,v)$-rigid proper induced subgraph $G[X]$ 
of $G$. To see this consider two edges $e,f\in E$ incident with $u$ and $v$, respectively.
Since $G$ is 
$3$-connected and redundantly rigid, it is
${\cal R}_2$-connected by Lemma \ref{rrmc}.
Hence there is an ${\cal R}_2$-circuit $C$ in $G$ with $e,f\in E(C)$.
Since $|E(C)|=2|V|-2$ and $d_C(v)\geq 3$ for all $v\in V(C)$, it follows that 
$C$ has at least four vertices of degree three. Thus 
there is a
vertex $w\in V(C)$ with $w\notin \{u,v\}$ and $d_C(w)=3$. Now $X=V(C)-w$ induces the desired
$(u,v)$-rigid subgraph.

In the rest of the proof we show that every non-adjacent vertex pair $\{u,v\}$ of $G$ is weakly globally
linked in $G$. The theorem will follow from this by Lemma \ref{J}(c).
Let us fix $u,v$ and consider a $(u,v)$-rigid proper induced subgraph $G[X]$ of $G$.
As we have shown above, such a subgraph exists.
By Theorem \ref{globlaza jellemzesd} it suffices to show that ${\rm Clique}(G,X)$ is globally
rigid. 

Let $D$ be the vertex set of a component of $G-X$ and let $H$ be obtained
from $G-D$ by adding a new edge $xy$ for each non-adjacent pair $x,y\in N_G(D)$. Since $G$ can be obtained from $H$ by attaching a graph along a complete
subgraph, and removing edges, the $3$-connectivity of $G$ implies that $H$
is $3$-connected. A similar argument shows that $H$ is rigid, and so is $H-e$
for every edge $e$ in $H$ not induced by $N_G(D)$. 
Thus if $H$ has some ${\cal R}_2$-bridges, then they are all induced by
$N_G(D)$. If $|N_G(D)|\geq 4$, then each edge induced by $N_G(D)$ belongs to
a $K_4$ in $H$, so $H$ cannot have ${\cal R}_2$-bridges at all.
If $|N_G(D)|=3$, then every
${\cal R}_2$-bridge in $H$ belongs to the same triangle, on the vertices of $N_G(D)$.
But that is impossible by Lemma \ref{rrtri}.
Therefore $H$ is a rigid graph with no ${\cal R}_2$-bridges,
and hence it is redundantly rigid.
By repeated applications of this argument we obtain that ${\rm Clique}(G,X)$ is $3$-connected and redundantly rigid. %
Since $|X|\leq |V|-1$, we can now use induction to deduce that
${\rm Clique}(G,X)$ is globally rigid.
This completes the proof.
\end{proof}

We remark that
a different proof for the sufficiency part in 
Theorem \ref{thm} was also given by Tanigawa \cite{Tani}. The high level ideas of his proof and the
proof given in this subsection are similar. By using our notation  
the main lemma \cite[Lemma 4.1]{Tani} can be stated as follows: if $v$ is a vertex of degree at least
$d+1$ in $G$,
$G-v$ is rigid in $\R^d$, and ${\rm Clique}(G,V-\{v\})$ is globally rigid in $\R^d$, then $G$ is globally rigid in $\R^d$. This statement is a 
special case of 
the “only if" direction of our Lemma \ref{cceq}.

\subsection{The second contraction lemma}\label{sec:secondlemma}

As a corollary of Lemma \ref{lem gllaza1}, it can be deduced that if $G[V_0]$ is a $(u,v)$-rigid subgraph of a graph $G$, $V_1$ is the vertex set of a component of $G-V_0$ and $\{u,v\}$ is weakly globally linked in $G/V_1$, then $\{u,v\}$ is weakly globally linked in $G$. In this subsection we shall prove the converse of this statement, see Lemma \ref{lem gllaza2} below.

We shall need some new notions and an auxiliary lemma.
A {\it configuration} of a set $U$ is a function that maps $U$ into $\R^d$. 
Two configurations $p_1,p_2$ of $U$ are said to be {\it congruent} if $||p_1(u)-p_1(v)||=||p_2(u)-p_2(v)||$ for all $u,v\in U$. Suppose that $p$ and $q$ are two incongruent configurations of a set $U$. 
We call a point $x\in \R^d$ {\it $(q,p)$-feasible} if there exists a point $y\in \R^d$ such that $||p(u)-x||=||q(u)-y||$ for all $u\in U$. We then call $y$ a {\it $(q,p)$-associate} of $x$. Observe that if $\pi$ is an isometry of $\R^d$, then the set of $(q,p)$-feasible points is equal to the set of $(\pi \circ q, p)$-feasible points.
The affine hull of set $X\subseteq \R^d$ will be denoted by Aff$(X)$.

\begin{lemma}
\label{not feasible point}
Let $p$ be a configuration of a set $U$. Suppose that $Q=\{q_1,\dots,q_k\}$ is a non-empty set of configurations of $U$ 
such that $q_i$ is not congruent to $p$, for all $1\leq i\leq k$.
Let $F_i$ be the set of $(q_i,p)$-feasible points, $1\leq i\leq k$. 
Then $\R^d-\bigcup_{i=1}^k F_i$
is a non-empty open set.
\end{lemma}

\begin{proof}
Let $S=\R^d-\bigcup_{i=1}^k F_i$.
We claim that $F_i$ is closed for every $i\in \{1,\dots, k\}$, which will imply that
$S$ is open.
Let $x_j\to x$ be a convergent sequence with $x_j\in F_i$, $j\in \mathbb{N}$,
and let $y_j$ be a $(q_i,p)$-associate of $x_j$. The set $\{y_j:j\in \mathbb{N}\}$ is bounded. Hence there exists a convergent subsequence $y_{j_{\ell}}\to y$. Then $y$ is a $(q_i,p)$-associate of $x$, which gives $x\in F_i$.
This proves the claim.

In the rest of the proof 
we show that $S$ is non-empty. Notice that $|U|\geq 2$ must hold.
We shall prove the following stronger statement by induction on $|U|$: for every $a\in U$ we have
\begin{equation}
\label{U}
S\cap \text{Aff}(p(U-\{a\}))\neq \emptyset
\end{equation}
First suppose that 
$|U|=2$, and let $U=\{a,b\}$. Then we have $||q_i(a)-q_i(b)||\neq ||p(a)-p(b)||$,
since $q_i$ and $p$ are not congruent. Thus $p(b)\in S$, and hence (\ref{U}) follows.   

Next suppose that $|U|\geq 3$. Let 
$$Q'=\{q\in Q: p|_{U-\{a\}}\ \hbox{is not congruent to}\ q|_{U-\{a\}}\},$$ and 
let $Q''=Q-Q'$.
By putting $F'=\bigcup_{q_i\in Q'}F_i$ and $F''=\bigcup_{q_i\in Q''}F_i$, we have
$S=\R^d-F'-F''$. 
By induction, the set ${\rm Aff}(p(U-\{a,b\}))-F'$ is non-empty for every $b\in U-\{a\}$. Since $F'$ is closed, 
this implies that
\begin{equation}
\label{Ueq1}    
{\rm Aff}(p(U-\{a\}))-F'\ \hbox{ is non-empty and relatively open in}\ {\rm Aff}(p(U-\{a\})).
\end{equation}
We claim that for all $q_i \in Q''$ the set
\begin{equation}
\label{Ueq2}    
F_i\cap {\rm Aff}(p(U-\{a\}))\ \hbox{ is either empty or a proper affine subspace of}\ {\rm Aff}(p(U-\{a\})).
\end{equation}
To prove the claim, let $q_i\in Q''$. By replacing $q_i$ with $\pi \circ q_i$, where $\pi$ is an appropriate isometry of $\R^d$, we may assume that $p|_{U-\{a\}}= q_i|_{U-\{a\}}$. Then it follows from the incongruency of $p$ and $q_i$ that $p(a)\neq q_i(a)$. Suppose that $x\in {\rm Aff}(p(U-\{a\}))$ and $y$ is a $(q_i,p)$-associate of $x$. Then there exists an isometry that fixes each point of $p(U-\{a\})$ and maps $x$ to $y$.
This isometry fixes each point of ${\rm Aff}(p(U-\{a\}))$, therefore, $y=x$. So the only possible $(q_i,p)$-associate of $x$ is $x$ itself.
It follows that $x$ is $(q_i,p)$-feasible if and only if $||x-p(a)||=||x-q_i(a)||$, that is, if $x$ is in the bisector hyperplane $H$ of $p(a)$ and $q_i(a)$.
Since $q_i$ and $p$ are not congruent, we obtain ${\rm Aff}(p(U-\{a\}))\not\subseteq H$. This proves the claim.

The lemma follows by noting that (\ref{Ueq1}) and (\ref{Ueq2}) yield that the set
$S\cap \text{Aff}(p(U-\{a\}))=\big({\rm Aff}(p(U-\{a\}))-F'\big)-F''$ is non-empty, and hence (\ref{U}) holds. 
\end{proof}

\begin{lemma}
\label{lem gllaza2}
Let $G=(V,E)$ be a graph, $u,v\in V$, and suppose that $G[V_0]$ is a $(u,v)$-rigid
subgraph of $G$.
Let $V_1$ be the vertex set of some component of $G-V_0$.
Then $\{u,v\}$ is weakly globally linked in $G$ in $\R^d$ if and only if $\{u,v\}$ is weakly globally linked in $G/V_1$ in $\R^d$.
\end{lemma}
\begin{proof} 
Since $G/V_1$ can be obtained from $G$ by contracting edges not induced by $V_0$, the “if" direction follows by repeated applications of Lemma \ref{lem gllaza1}.
To prove the “only if" direction suppose that $\{u,v\}$ is weakly globally linked in $G$ and let $(G,p)$ be a generic realization of $G$ in which $\{u,v\}$ is globally linked.
Let $v_1$ be the vertex of $G/V_1$ obtained by the contraction of $V_1$ in $G$.
We shall prove that $p|_{V-V_1}$ has an extension to $(V-V_1)\cup \{v_1\}$ that is a generic realization of  $G/V_1$
in which $\{u,v\}$ is globally linked.
We may assume that $\{u,v\}$ is not globally linked in $(G-V_1,p|_{V-V_1})$, for otherwise we are done by choosing an arbitrary generic extension.

Let $q_1,\dots,q_k$ be a maximal set of pairwise incongruent configurations of $V_0$ 
such that
$||q_i(u)-q_i(v)||\neq ||p(u)-p(v)||$ and
$q_i$ is a restriction of some realization of $G-V_1$ which is equivalent to $(G-V_1,p|_{V-V_1})$,
 for $1\leq i\leq k$.
By our assumption $k\geq 1$. 
Proposition \ref{finiteequi} implies
that $k$ is finite, since 
$G[V_0]$ is rigid, $p$ is generic and $(G[V_0],q_i)$ is equivalent to $(G[V_0],p|_{V_0})$.
For all $i\in \{1,\dots,k\}$ the configurations $q_i|_{N_G(V_1)}$ and $p|_{N_G(V_1)}$ are incongruent, for otherwise
$q_i$ would be extendible to a configuration $q_i'$ so that $(G,q_i')$ is equivalent to $(G,p)$, contradicting the assumption that $\{u,v\}$ is globally linked in $(G,p)$.

Applying Lemma \ref{not feasible point} to $N_G(V_1)$, $p|_{N_G(V_1)}$ and 
the set $Q=\{q_1|_{N_G(V_1)},\dots, q_k|_{N_G(V_1)}\}$  gives that there is some $x=(x_1,\dots, x_d)\in \R^d$ for which $x$ is not $(q_i|_{N_G(V_1)},p|_{N_G(V_1)})$-feasible for all $i\in \{1,\dots,k\}$ and for which $p(V-V_1)\cup \{x\}$ is generic. We can now complete the proof of the lemma by considering the generic realization $(G/V_1,\overline{p})$, where
$\overline{p}|_{V-V_1}=p|_{V-V_1}$ and $\overline{p}(v_1)=x$.
Then 
$\{u,v\}$ is globally linked in $(G/V_1,\overline{p})$.
Indeed, the existence of an equivalent realization $(G/V_1,q)$ with $||q(u)-q(v)||\neq ||\overline{p}(u)-\overline{p}(v)||$ would imply that $q|_{V_0}=q_i$ for some $1\leq i\leq k$ and that
$x$ is $(q_i|_{N_G(V_1)},p|_{N_G(V_1)})$-feasible, contradicting the choice of $x$.
\end{proof}

Let $G=(V,E)$ be a graph, $u,v\in V$, and let $G[V_0]$ be a $(u,v)$-rigid subgraph of $G$.
Let $e=(s_1,s_2)\in E$ be an edge with $s_1,s_2\notin V_0$. Notice that Lemma \ref{lem gllaza2} implies that $\{u,v\}$ is weakly globally linked in $G$ in $\R^d$ if and only if $\{u,v\}$ is weakly globally linked in $G/e$ in $\R^d$: each of these two conditions is equivalent to the condition that $\{u,v\}$ is weakly globally linked in $G/V_1$, where $V_1$ is the connected component of $G-V_0$ that contains $e$. By the same argument, for any connected subgraph $G_1$ of $G-G_0$,  $\{u,v\}$ is weakly globally linked in $G$ in $\R^d$ if and only if $\{u,v\}$ is weakly globally linked in $G/V(G_1)$ in $\R^d$.

\section{Weakly globally linked pairs in $\R^2$}
\label{sec:2d}

In this section we focus on the $d=2$ case. 
Thus, we shall occasionally write that a graph is (globally) rigid to mean that it is (globally) rigid in $\R^2$, and we may similarly omit the dimension when referring to global linkedness of vertex pairs in graphs.
This section contains one of our main results, a characterization of weakly globally
linked pairs in graphs.

\subsection{Weakly globally linked pairs in $3$-connected graphs}

We start with the special case of $3$-connected graphs. By Lemma \ref{notlinked} it suffices to consider 
non-adjacent
linked pairs $\{u,v\}$ of $G$, or equivalently,
pairs $\{u,v\}$ for which there exists some subgraph $G_0=(V_0,E_0)$ of $G$ with $u,v\in V_0$ such that $G_0+uv$ is an ${\cal R}_2$-circuit.

\begin{theorem}
\label{globlaza jellemzes2}
Let 
$G=(V,E)$ be a $3$-connected graph and $u,v\in V$ with $uv\notin E$.
Suppose that $G_0=(V_0,E_0)$ is a subgraph of $G$ with $u,v\in V_0$ such that $G_0+uv$ is an ${\cal R}_2$-circuit. Then 
$\{u,v\}$ is weakly globally linked in $G$ in $\R^2$ if and only if 
${\rm Clique}(G,V_0)$ is globally rigid in $\R^2$.
\end{theorem}

\begin{proof}
Since $G[V_0]$ is rigid, sufficiency follows from Theorem \ref{globlaza jellemzesd}.
To prove the other direction
suppose, for a contradiction, that $\{u,v\}$ is weakly globally linked in $G$
but ${\rm Clique}(G,V_0)$ is not globally rigid.
Since $G_0+uv$ is redundantly rigid,
so is 
${\rm Clique}(G,V_0)+uv$. The $3$-connectivity of $G$ implies
that ${\rm Clique}(G,V_0)$ is $3$-connected. Thus 
${\rm Clique}(G,V_0)+uv$ is globally rigid by Theorem \ref{thm}.
Hence $\{u,v\}$ is globally loose in ${\rm Clique}(G,V_0)$ by Lemma \ref{critical}.
As $G$ can be obtained from ${\rm Clique}(G,V_0)$ by clique sum operations and removing edges,
Lemma \ref{cliquesum} implies that $\{u,v\}$ is globally loose in $G$, a contradiction.
\end{proof}

\subsection{Weakly globally linked pairs and 2-separators}

In this subsection we shall prove some lemmas that 
describe, among others, how weak global linkedness is affected when
the graph is cut into two parts along a separating pair of vertices.
These lemmas will enable us to reduce the question of
whether a linked pair of vertices in a graph $G$ is weakly globally linked to the case
when $G$ is $3$-connected.
We shall also need the following extension of \cite[Corollary 5]{JJS2}.

\begin{lemma}
\label{glob laza elegs min merev}
Let $G=(V,E)$ be a rigid graph, $ab\in E$ an ${\cal R}_2$-bridge in $G$, 
and $u,v\in V$ with $uv\notin E$.
Suppose that 
$G$ has no $(u,v)$-rigid proper induced subgraph. 
Then $\{u,v\}$ is globally loose in $G$.
\end{lemma}

\begin{proof}
Let $H=(V,B)$ be a minimally rigid spanning subgraph of $G$. Since $ab$ is an ${\cal R}_2$-bridge, we have
$ab\in B$. 
The graph $H+uv$ contains  a unique ${\cal R}_2$-circuit $C$ with $u,v\in V(C)$. Since
$G$ has no $(u,v)$-rigid proper induced subgraph, $C=H+uv$ must hold. 
Let $(G,p_0)$ be a generic realization of $G$.
By \cite[Lemma 11]{JJS2}
the generic framework
$(H-ab,p_0)$ has an equivalent realization $(H-ab,p_1)$, which can be obtained by a flexing, and for which
$||p_0(u)-p_0(v)||\neq ||p_1(u)-p_1(v)||$ and $||p_0(a)-p_0(b)||= ||p_1(a)-p_1(b)||$.
Consider an edge $xy\in E-B$. Since $H$ is rigid, $xy$ belongs to an ${\cal R}_2$-circuit $C'$ of $H+xy$.
Moreover, 
$ab$ is an ${\cal R}_2$-bridge in $G$ (as well as in its subgraph $H+xy$), and hence $C'$ does not contain $ab$.
Thus there is a 
rigid subgraph of $H-ab$ (namely, $C'-xy$), which contains $x$ and $y$.
Hence the flexing does not
change the distance between $x$ and $y$. Therefore
$(G,p_1)$ is equivalent to $(G,p_0)$. Since the distances between $u$ and $v$ are different in
these realizations, it follows that $\{u,v\}$ is globally loose in $G$.
\end{proof}

\begin{lemma}
\label{glob laza masodfoku csucs}
Let
$G=(V,E)$ be a rigid graph, let $z\in V$ with
$N_G(z)=\{x,y\}$, and let $u,v\in V-\{z\}$. Then $\{u,v\}$ 
is weakly globally linked in $G-z+xy$ if and only if
$\{u,v\}$ is weakly globally linked in $G$.
\end{lemma}

\begin{proof}
Let $G_1=G-z+xy$.
Observe that $G_1$ is isomorphic to $G/zx$. Since
$G-z$ is rigid, we can use Lemma \ref{lem gllaza1} to deduce that if $\{u,v\}$ 
is weakly globally linked in $G_1$, then $\{u,v\}$ is weakly globally linked in $G$.
To prove the other direction suppose that $\{u,v\}$ is weakly globally linked in $G$.
Then it is also weakly globally linked in $G+xy$.
Since $G+xy$ is the clique sum of $G_1$ and a copy of $K_3$,
Lemma \ref{cliquesum} implies that $\{u,v\}$ is weakly globally linked in $G_1$.
\end{proof}

A pair $(a,b)$ of vertices of a $2$-connected graph $H=(V,E)$ is called
a {\it $2$-separator} if $H-\{a,b\}$ is disconnected.

\begin{lemma}
\label{nem 2of visszavezetes}
Let $G=(V,E)$ be a rigid graph with $|V|\geq 4$ and
$(a,b)$ be a $2$-separator in $G$. 
Let $C$ be a connected component of $G-\{a,b\}$ and
let $V_0=V(C)\cup \{a,b\}$. Suppose that $u,v\in V_0$.
Then $\{u,v\}$ is weakly globally linked in $G$ if and only if $\{u,v\}$ is weakly globally linked in $G[V_0]+ab$.
\end{lemma}

\begin{proof}
If $\{u,v\}$ is weakly globally linked in $G$, then it is easy to see, by using
Lemma \ref{cliquesum}, that $\{u,v\}$ is weakly globally linked in $G[V_0]+ab$.

To prove that if $\{u,v\}$ is weakly globally linked in $G[V_0]+ab$, then $\{u,v\}$ is weakly globally linked in $G$,
we use induction on $|V|$. If $|V|=4$, then we must have $G=K_4-e$ and $uv\in E$, so the statement is obvious.
Suppose that $|V|\geq 5$.
If there exists a $(u,v)$-rigid subgraph of $G[V_0]$, then, since
$G[V_0]+ab$ can be obtained from $G$ by a sequence of edge contractions, we can use
Lemma \ref{lem gllaza1} to deduce that $\{u,v\}$ is weakly globally linked in $G$.
So in the rest of the proof we may assume that
\begin{equation}
\label{no}
\hbox{$G[V_0]$ has no $(u,v)$-rigid subgraph.}
\end{equation}
In particular, $G[V_0]$ is not rigid. Hence, by the rigidity of $G$,
it follows that $\{a,b\}$ is not linked in $G[V_0]$ and
$ab$ is an ${\cal R}_2$-bridge in $G[V_0]+ab$.

Since $\{u,v\}$ is weakly globally linked in $G[V_0]+ab$, Lemma \ref{glob laza elegs min merev} implies that
there exists a $(u,v)$-rigid proper induced subgraph
$G'=(V',E')$ of $G[V_0]+ab$. 
Suppose that $G'$ is vertex-minimal.
By (\ref{no}) we obtain $ab\in E'$ and $a,b\subset V'$.
We consider three cases depending on the structure of
$G[V_0]-V'$.
Since $G'$ is a proper induced subgraph, we have $V_0-V'\not= \emptyset$. See Figure \ref{fig:gen9}.
\begin{figure}[ht]
\centering
\includegraphics[scale=0.9]{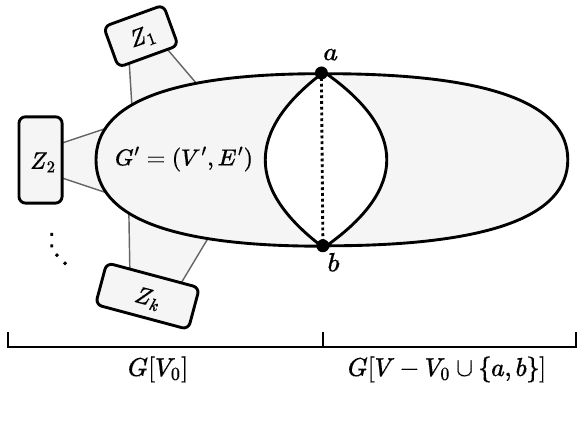}
\vspace{-0.6cm}
\caption{The case where $G[V_0]$ has no $(u,v)$-rigid subgraph.}
\label{fig:gen9}
\end{figure}
\medskip 

\noindent \textbf{Case 1:} $G[V_0]-V'$ has a component $Z$ with $|V(Z)|\geq 2$. 
\medskip

By Lemma \ref{lem gllaza2} $\{u,v\}$ is weakly globally linked in $(G[V_0]+ab)/Z$. 
Since $G$ and $G'$ are rigid, $G-V(Z)$ is also rigid, and $Z$ has at least two neighbours in $G$. 
Hence
$G/Z$ is rigid.
Thus we obtain, by induction, that
$\{u,v\}$ is weakly globally linked in $G/Z$. By using that $G-V(Z)$
is rigid, Lemma \ref{lem gllaza2}
gives that $\{u,v\}$ is weakly globally linked in $G$.
\medskip

\noindent \textbf{Case 2:} Each component of $G[V_0]-V'$ is a singleton and there exists a vertex $z\in V_0-V'$ with $d_G(z)=2$.
\medskip

Let $N_G(z)=\{x,y\}$. By Lemma \ref{glob laza masodfoku csucs} $\{u,v\}$ is weakly globally linked in $(G[V_0]+ab)-z+xy$. 
If $\{u,v\}=\{a,b\}$ and $|V_0|=3$, then $\{u,v\}$ is weakly globally linked in $G$
by Lemma \ref{lem gllaza1}.
So we may assume that $|V_0|\geq 4$, and hence
$(a,b)$ is a $2$-separator of the rigid graph $G-z+xy$. 
Hence
$\{u,v\}$ is weakly globally linked
in $G-z+xy$ by induction. 
By using that $G$ is rigid,
Lemma \ref{glob laza masodfoku csucs} implies that $\{u,v\}$ is weakly globally linked
in $G$.
\medskip
 
\noindent \textbf{Case 3:} Each component of $G[V_0]-V'$ is a singleton and for each $z\in V_0-V'$ we have $d_G(z)\geq 3$. 
\medskip

We claim that for each $z\in V_0-V'$ and $x,y\in N_G(z)$ 
there is a rigid subgraph of $G'-ab$ which contains $x$ and $y$.
To see this let $w$ be another neighbour of $z$, different from $x,y$,
and let $G''$ be obtained from $G'-ab$ by adding vertex $z$ and edges $zx,zy,zw$.
The three edges incident with $z$ in $G''$ cannot be ${\cal R}_2$-bridges, since it
would imply, by using the rigidity of $G'$ and computing ranks, 
that $G''$ is $(u,v)$-rigid, contradicting (\ref{no}).
Thus there is an ${\cal R}_2$-circuit $C$ in $G''$ containing $z$. Then $C$ must contain
$x$ and $y$, too, and $C-z$ is a rigid subgraph of $G'-ab$ which contains $x$ and $y$, as claimed.

The minimality of $G'$ implies that it has no $(u,v)$-rigid proper induced subgraph.
Let $(G[V_0]+ab,p)$ be a generic realization.
By (the proof of) Lemma \ref{glob laza elegs min merev} $(G'-ab,p|_{V'})$ has an equivalent realization $(G'-ab,q)$, for which $||p(u)-p(v)||\neq ||q(u)-q(v)||$,  $||p(a)-p(b)||= ||q(a)-q(b)||$, and such
that the distances between the linked pairs of $G'-ab$ are the same in the two realizations.
Then, since 
each pair of neighbours of every $z\in V_0-V'$ is
linked in
$G'-ab$, it follows that $(G'-ab,q)$
can be  extended to a realization $(G[V_0]+ab,q')$ that is equivalent to $(G[V_0]+ab,p)$.
Hence $\{u,v\}$ is
globally loose in $G[V_0]+ab$, a contradiction.  This completes the proof.
\end{proof}

We next extend Lemma \ref{nem 2of visszavezetes} to 2-connected graphs.

\begin{lemma}\label{linkedpairs and 2separ}
Let $G=(V,E)$ be a 2-connected graph and $\{u,v\}$ be a linked pair of vertices of $G$. Suppose that $(a,b)$ is a 2-separator of $G$. Let $C$ be a connected component of $G-\{a,b\}$, and let $V_0=V(C)\cup \{a,b\}$. Suppose that $u,v\in V_0$. Then $\{u,v\}$ is weakly globally linked in $G$ if and only if $\{u,v\}$ is weakly globally linked in $G[V_0]+ab$.
\end{lemma}

\begin{proof}
 If $\{u,v\}$ is weakly globally linked in $G$, then it follows from Lemma \ref{cliquesum} that $\{u,v\}$ is weakly globally linked in $G[V_0]+ab$. To prove the “if" direction suppose that $\{u,v\}$ is weakly globally linked in $G[V_0]+ab$. Since $\{u,v\}$ is a linked pair, there is a
$(u,v)$-rigid induced subgraph $G[U]$ of $G$. If $\{a,b\}\not\subseteq U$, then $U$ is a subset of $V_0$
and $G[V_0]+ab$ can be obtained from $G$ by contracting edges which are not induced by $U$. Thus
$\{u,v\}$ is weakly globally linked in $G$ by Lemma \ref{lem gllaza1}. So we may suppose that $\{a,b\}\subseteq U$. Let $A_1,\dots, A_k$ be the components of $G-U$ contained in $V_0$, and let $B_1, \dots, B_l$ be the components of $G-U$ not contained in $V_0$.
Observe that the rigidity of $G[U]$ and the 2-connectivity of $G$ imply that
$G/A_1/\dots /A_k/B_1/\dots/ B_k$ is rigid. Hence we have that
\vspace{-0 cm}
\begin{align*}
&\{u,v\} \text{ is weakly globally linked in } G \\ 
\;\;\Leftrightarrow\;\; &\{u,v\} \text{ is weakly globally linked in } G/A_1/\dots /A_k/B_1/\dots/ B_k \\
\;\;\Leftrightarrow\;\;&\{u,v\} \text{ is weakly globally linked in } G[V_0]/A_1/\dots/A_k + ab\\ 
\;\;\Leftrightarrow\;\; &\{u,v\} \text{ is weakly globally linked in } G[V_0]+ab,
\end{align*} 
where the first and third equivalence follows from Lemma \ref{lem gllaza2} and the second equivalence follows from Lemma \ref{nem 2of visszavezetes}, using the rigidity of $G/A_1/\dots /A_k/B_1/\dots/ B_k$.
\end{proof}

The next lemma on the weak global linkedness of linked separating pairs
follows from
Lemma \ref{linkedpairs and 2separ} by putting $\{a,b\}=\{u,v\}$.
It can also be deduced from Lemma \ref{lem gllaza1} by using that 
there is some component $C$ of $G-\{u,v\}$ for which $\{u,v\}$ is linked in $G[V(C)\cup \{u,v\}]$.

\begin{lemma}
\label{pair}
Let $G=(V,E)$ be a 2-connected graph, and $u,v\in V$ be a linked pair of vertices for which $(u,v)$ is a 2-separator in $G$. Then $\{u,v\}$ is weakly globally linked in $G$.
\end{lemma}

We use the following operation to eliminate 2-separators. Let $G=(V,E)$ be a 2-connected graph, let 
$(a,b)$ be a 2-separator in $G$, and let 
$C$ be a connected component of $G-\{a,b\}$.
We say that the graph $G[V(C)\cup \{a,b\}]+ab$ (when $ab\notin E$) or  $G[V(C)\cup \{a,b\}]$ (when $ab\in E$)
is obtained from $G$ by a
{\it cleaving operation} along $(a,b)$. The graph $\bar G$ obtained from $G$ by adding every edge $ab$, for which $ab\notin E$ and $(a,b)$ is a 2-separator of $G$,
is called the {\it augmented graph} of $G$.

The following lemma is easy to show by induction, using the cleaving operation.

\begin{lemma}
\label{lem:cl}
Let $G=(V,E)$ be a 2-connected graph and let $\{u,v\}$ be a non-adjacent vertex pair in $G$
with $\kappa_G(u,v)\geq 3$. Then either $(u,v)$ is a separating pair in $G$ or there is a unique
maximal 3-connected subgraph $B$ of $\bar G$ with $\{u,v\}\subset V(B)$.
In the latter case
the subgraph $B$ can be obtained from $G$ by a sequence of cleaving operations.
Furthermore, $uv\notin E(B)$, and
if the pair $\{u,v\}$ is linked in $G$ then it is also linked in $B$.
\end{lemma}

The subgraph $B$ in Lemma \ref{lem:cl} is called the {\it 3-block} of $\{u,v\}$ in $G$.

We are ready to state the 
main result of this section: a complete characterization of the non-adjacent weakly globally linked pairs in a graph $G$.
By Lemma \ref{notlinked} and Lemma \ref{kappa}
we may assume that $\{u,v\}$ is linked and $\kappa_G(u,v)\geq 3$ (for otherwise $\{u,v\}$ is globally loose).
By Lemma \ref{cliquesum} we may also assume that
$G$ is 2-connected.

\begin{theorem}\label{cleavunit}
Let $G=(V,E)$ be a 2-connected graph and let $\{u,v\}$ be a non-adjacent linked pair of vertices
with $\kappa_G(u,v)\geq 3$.
Then $\{u,v\}$ is weakly globally linked in $G$ if and only if either\\
(i) $(u,v)$ is a separating pair in $G$, or\\
(ii) ${\rm Clique}(B,V_0)$ is globally rigid,\\
where $B$ is the 3-block of $\{u,v\}$ in $G$, and $B_0=(V_0,E_0)$ is a subgraph
of $B$ with $u,v\in V_0$ such that $B_0+uv$ is an ${\cal R}_2$-circuit.
\end{theorem}

\begin{proof}
The proof is by induction on the number $h$ of vertex pairs $x,y\in V$ with $\kappa_G(x,y)=2$.
If $h=0$, then $B=G$ and (ii) holds by Theorem \ref{globlaza jellemzes2}.
Suppose that $h\geq 1$ and let $(a,b)$ be a 2-separator in $G$. If $\{a,b\}=\{u,v\}$ then
Lemma \ref{pair} applies and (i) holds.  
Otherwise we can use Lemmas \ref{linkedpairs and 2separ}, \ref{lem:cl}, and induction, to
complete the proof.%
\end{proof}

See Figure \ref{fig:gen10} for an illustration of
Theorem \ref{cleavunit}.

\begin{figure}[ht]
\centering
\includegraphics[scale=1.25]{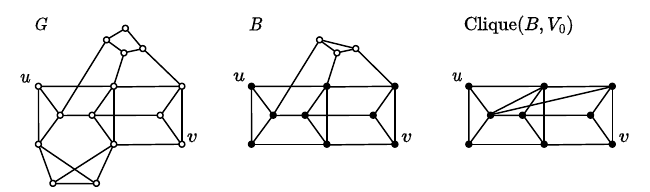}
\caption{Consider the graph $G$. The pair $\{u,v\}$ is linked in $G$, as there is a $(u,v)$-rigid subgraph of $G$, and we have $\kappa_G(u,v)=3$. $B$ is the 3-block of $\{u,v\}$ in $G$. Consider the subgraph $B_0=(V_0,E_0)$ of $B$ that is induced by the solid vertices. $B_0+uv$ is an ${\cal R}_2$-circuit. Since ${\rm Clique}(B,V_0)$ is globally rigid, $\{u,v\}$ is weakly globally linked in $G$ by Theorem \ref{cleavunit}.}
\label{fig:gen10}
\end{figure}

\section{Concluding remarks}\label{section:concluding}

\subsection{Algorithmic aspects}

Theorem \ref{cleavunit} and its proof shows that weak global linkedness of a vertex pair
$\{u,v\}$ in a graph $G=(V,E)$ can be tested in $O(|V|^2)$ time, as efficient
algorithms are available for each of the required subroutines. Basic graph algorithms can
be used to test, in linear time, whether $\kappa_G(u,v)\geq 3$ holds and to find the maximal
2-connected block that contains $u,v$. After reducing the problem to the 2-connected case,
the linear time algorithm of \cite{HT} can be applied to check whether $(u,v)$ is a separating pair and (when it is not)
to identify the 3-block $B$ of $\{u,v\}$. (Note that $B$ coincides with one of the
so-called cleavage units of $G$.) Computing ${\rm Clique}(G,X)$ for a given $X\subseteq V$ is also easy.

Testing whether $\{u,v\}$ is linked, and (when it is linked) finding an ${\cal R}_2$-circuit of $G+uv$ containing $uv$
can be done in $O(|V|^2)$ time \cite{BJ}. Within the same time bound, we can test
whether a graph is globally rigid, see e.g. \cite{BJ,KM}.

\subsection{Higher dimensions}

Most questions concerning 
the higher dimensional versions of (weak) global linkedness are open.
Partial results can be found in \cite{GJpartial,JKT, JT}.
A fairly natural new question is whether 
the sufficient condition of weak
global linkedness given in Theorem \ref{globlaza jellemzesd} is also necessary for $d\geq 3$, in the sense that for every weakly globally linked pair $\{u,v\}$ of a graph $G$ there exists a $(u,v)$-rigid induced subgraph $G[X]$ of $G$ such that ${\rm Clique}(G,X)$ is globally rigid. We are not aware of
any counter-examples.

Finding extensions and stronger versions of our results is another promising research
direction. Here we mention one example. 
If we use
Lemma \ref{linkedext} (suggested by D. Garamv\"olgyi) in place of Proposition \ref{finiteequi} in the proof, we obtain 
the following strengthening of Lemma \ref{lem gllaza1}
to linked pairs: if $G[V_0]$ is a subgraph of $G$ in which $\{u,v\}$ is linked, $e\in E-E(G[V_0])$ and $\{u,v\}$ is weakly globally linked in $G/e$, then $\{u,v\}$ is weakly globally linked in $G$.
Note that
a pair $\{u,v\}$ may be linked in a graph $G_0$ in $\R^d$, $d\geq 3$, even if $G_0$ contains no $(u,v)$-rigid subgraph. 
For the definitions of the new notions appearing in the next proof see e.g. \cite{JW,SW}.

\begin{lemma}
\label{linkedext}
Let $\{u,v\}$ be a linked pair in a graph $G$ in $\R^d$ and let $(G,p)$ be a generic realization of $G$
in $\R^d$. Then the set
$$\{ ||q(u) - q(v)|| : (G,q) \ \hbox{is equivalent to}\ (G,p) \}$$
is finite.
\end{lemma}

\begin{proof} Suppose, for a contradiction, that there exists an infinite sequence of frameworks
$(G,q_i)$, $i\geq 1$, equivalent to $(G,p)$, in which the 
distances $||q_i(u) - q_i(v)||$ are pairwise different.
We may assume that $G$ is connected and $q_i(u)$ is the origin for all $i\geq 1$.
Then each $(G,q_i)$ is in the interior of a ball of radius $K$, for some constant $K$.
Thus, by choosing a subsequent, if necessary, we may assume that $(G,q_i)$ is convergent, with limit $(G,q)$. 
Since $(G,q)$ is equivalent to $(G,p)$, and $(G,p)$ is generic, the two frameworks
$(G,p)$ and $(G,q)$ have the same 
equilibrium stresses by \cite{Con}. 
In particular, the rank of the rigidity matrix of $(G,q)$ is equal to the maximum (generic) rank of $G$.   
This fact, and the linkedness of $\{u,v\}$ imply that the ranks of the rigidity matrices
of $(G+uv,q)$ and $(G,q)$ are the same.
So their kernels are the same, too. Thus every infinitesimal motion $x:V\to \R^{d}$ of $(G,q)$ satisfies
$(q(u)-q(v))^T(x(u)-x(v)) = 0$. By continuity this %
holds for all frameworks in a small enough
neighbourhood of $(G,q)$. 
Consider the frameworks
$q'_i = (q_{i+1} + q_i)/2$. They converge to $q$, and the well-known avaraging technique shows that $x_i = (q_{i+1} - q_i)$ is an infinitesimal motion of $q'_i$ for all $i\geq 1$ (for a proof see e.g. \cite[Theorem 13]{coning}).
The same calculations show that, since $||q_{i+1}(u)-q_{i+1}(v)||\neq ||q_i(u)-q_i(v)||$, we have $(q'_i(u) - q'_i(v))^T(x_i(u)-x_i(v))\not= 0$, a contradiction. %
\end{proof}

\subsection{Minimally globally rigid graphs}

A graph $G=(V,E)$ is called {\it minimally globally rigid} in $\R^d$ if it is globally rigid in $\R^d$ and for every edge $e\in E$ the graph $G-e$ is not globally rigid in $\R^d$.
Garamv\"olgyi and Jord\'an \cite[Theorem 3.3]{GJmgr} proved that 
if $G=(V,E)$ is minimally globally rigid in $\R^d$ and $|V|\geq d+1$, then $$|E|\leq (d+1)|V|-\binom{d+2}{2}.$$ 
Moreover, as it is noted in \cite{GJmgr}, 
for every globally rigid graph $G$ in $\R^d$ on at least $d+1$ vertices, and for 
every minimally rigid spanning subgraph $G_0$ of $G$, there exists a globally rigid spanning 
subgraph of $G$ that contains $G_0$ and has at most $(d+1)|V|-\binom{d+2}{2}$ edges. 

Furthermore, the authors conjecture that a minimally globally rigid
graph in $\R^d$ is in fact ${\cal R}_{d+1}$-independent, see \cite[Conjecture 4.2]{GJmgr}. 
The truth of this
conjecture would imply that a minimally globally rigid graph $G=(V,E)$ in $\R^d$ is not only sparse, but every subgraph of $G$ is sparse:
for each $U\subseteq V$ with $|U|\geq d+1$ we have $|E(U)|\leq (d+1)|U|- \binom{d+2}{2}$.
This conjecture was verified for $d=2$ in \cite{GJmgr}.

Next we prove this upper bound for all $d$, 
in the special case when the subgraph induced by $U$ is rigid.

\begin{theorem}\label{mgr:thmrigid}
Let $G=(V,E)$ be a minimally globally rigid graph in $\R^d$. Suppose that $U\subseteq V$, $|U|\geq d+1$ and $G[U]$ is rigid. Then $|E(U)|\leq (d+1)|U|- \binom{d+2}{2}$.
\end{theorem}
\begin{proof}
Let $G_0=(U,E_0)$ be a minimally rigid spanning subgraph of $G[U]$. Since $G$ is globally rigid, so is ${\rm Clique}(G,U)$. Thus, by the results of \cite{GJmgr}, there is a globally rigid spanning subgraph $G'=(U,E')$ of ${\rm Clique}(G,U)$ that contains $G_0$ and has at most $(d+1)|U|-\binom{d+2}{2}$ edges. Suppose, for a contradiction, that there is some edge $e=uv\in E(U)-E'$. 
Note that $G[U]-e$ is rigid.
Then $G'$ is a subgraph of ${\rm Clique}(G-e,U)$, and hence $\{u,v\}$ is weakly globally linked in $G-e$ by Theorem \ref{globlaza jellemzesd}.
Since $e$ is critical in $G$, this contradicts Lemma \ref{critical}. It follows that $G[U]$ is a subgraph of $G'$; therefore $|E(U)|\leq |E'|\leq (d+1)|U|-\binom{d+2}{2}$.
\end{proof}

For $d=2$ we can extend Theorem \ref{mgr:thmrigid} to all subsets $U\subseteq V$ with $|U|\geq d+1$.
As we noted above, this two-dimensional result is not new, as it follows from 
\cite[Theorem 4.7]{GJmgr}. Here we give a new proof in order to illustrate how Theorem \ref{mgr:thmrigid} might be applied to attack the $d$-dimensional case.

\begin{theorem}\label{mgr2dim} \cite{GJmgr}
Let $G=(V,E)$ be a minimally globally rigid graph in $\R^2$. Suppose that $U\subseteq V$ and $|U|\geq 3$. Then $|E(U)|\leq 3|U|- 6$.
\end{theorem}
\begin{proof}
By Theorem \ref{mgr:thmrigid} the statement is true if $G[U]$ is rigid. Suppose that $G[U]$ is not rigid, that is, $r_2(G[U])\leq 2|U|-4$. 
We may assume that $G[U]$ has no isolated vertices.
It is well-known (see e.g. \cite{Jmemoirs}) that 
for the collection $G_i=(V_i,E_i)$, $1\leq i\leq k$, of the maximal rigid subgraphs of $G[U]$ we have
$\sum_{i=1}^k (2|V_i|-3)=r_2(G[U])$. 
For an integer $h\geq 2$ let
$f(h) = 3h-6$, if $h\geq 3$, and let $f(h)=1$ otherwise.
Then we have
 $$|E(U)|= \sum_{i=1}^k |E_i| \leq \sum_{i=1}^k f(|V_i|)\leq  \sum_{i=1}^k \frac{3}{2}(2|V_i|-3)=\frac{3}{2}r_2(G[U])\leq 3|U|-6,$$ where the first inequality follows from Theorem \ref{mgr:thmrigid}.
\end{proof}

\section{Acknowledgements}

This research has been implemented with the support provided by the Ministry of Innovation and Technology of Hungary from the National Research, Development and Innovation Fund, financed under the  ELTE TKP 2021-NKTA-62 funding scheme. The first author was also 
supported by the Hungarian Scientific Research Fund grant no. K135421, and
the MTA-ELTE Momentum Matroid Optimization Research Group.
 
We thank D\'aniel Garamv\"olgyi for several useful remarks, and for suggesting Lemma \ref{linkedext}. We also thank Csaba Kir\'aly for his comments.


\begin{thebibliography}{99}
\bibitem{AR} {\scshape L. Asimow and B. Roth}, The rigidity of graphs, {\it
  Trans. Amer. Math. Soc.}, 245 (1978), pp. 279-289.

\bibitem{BJ} {\scshape A.R. Berg and T. Jord\'an}, Algorithms for graph rigidity and scene analysis,
Proc. 11th Annual
European Symposium on Algorithms (ESA) 2003, (G. Di Battista and U. Zwick, eds) Springer
LNCS 2832, pp. 78-89, 2003.




\bibitem{Con} {\scshape R. Connelly},
\newblock
Generic global rigidity,
\newblock {\itshape Discrete Comput. Geom.} 33:549-563 (2005).

\bibitem{Conmerge} {\scshape R. Connelly},
Combining globally rigid frameworks,
Proc. of the Steklov Institute of Mathematics,
275, 191-198, 2011.

\bibitem{coning} {\scshape R. Connelly and W. Whiteley},
Global rigidity: the effect of coning,
{\itshape Discrete Comput Geom} (2010) 43: 717–735.





 
\bibitem{GJunitball}
{\scshape D. Garamv\"olgyi and T. Jord\'an}, Global rigidity of unit ball
graphs, {\it SIAM J. Discrete Math.} 34:1, pp. 212-229, 2020.

\bibitem{GJcccg}
{\scshape D. Garamv\"olgyi and T. Jord\'an}, Globally linked pairs in
braced maximal outerplanar graphs, Proc. CCCG 2022, Toronto, August 2022, pp. 162-168.

\bibitem{GJpartial}
{\scshape D. Garamv\"olgyi and T. Jord\'an}, Partial reflections and globally linked
pairs in rigid graphs, arXiv:2305.03412, May 2023.



\bibitem{GJmgr} {\scshape D. Garamv\"olgyi and T. Jord\'an}, Minimally globally
rigid graphs, {\it European J. Combin.}, Vol. 108., 103626, 2023.



\bibitem{Gluck} {\scshape H. Gluck},
\newblock
Almost all simply connected closed surfaces are rigid,
\newblock
{\itshape Geometric topology (Proc. Conf., Park City, Utah, 1974)}, 
pp. 225--239. 
Lecture Notes in Math., Vol. 438, 
Springer, Berlin, 1975.




\bibitem{GHT} {\sc S. Gortler, A. Healy, and D. Thurston},
{Characterizing generic global rigidity},
{\it American Journal of Mathematics}, Volume 132, Number 4, August 2010,
pp. 897-939.





\bibitem{hend} {\scshape B. Hendrickson},
\newblock Conditions for unique graph realizations,
\newblock {\itshape SIAM J. Comput.} {21} (1992), no. 1, 65-84.

\bibitem{HT} {\scshape J.E. Hopcroft and R.E. Tarjan}, Dividing a graph into triconnected
components, {\it SIAM J. Comput.} 2 (1973), 135--158.

\bibitem{JJconnrig} {\scshape B. Jackson and T. Jord\'an,}
\newblock Connected rigidity matroids and unique
realizations of graphs, 
{\it J. Combin. Theory Ser. B}, Vol. 94, 1-29, 2005.

\bibitem{JJS} {\sc B. Jackson, T. Jord\'an, and Z. Szabadka},
Globally linked pairs of vertices in equivalent realizations of
graphs, {\itshape Discrete Comput. Geom.}, Vol. 35,
493-512, 2006.

\bibitem{JJS2} {\sc B. Jackson, T. Jord\'an, and Z. Szabadka},
Globally linked pairs of vertices in rigid frameworks, in:
Rigidity and Symmetry,
{\it Fields Institute Communications}, Vol. 70,
R. Connelly, A. Ivic Weiss, W. Whiteley (Eds.) 2014, pp. 177-203.

\bibitem{Jmemoirs}{\sc T. Jord\'an},
Combinatorial rigidity: graphs and matroids
in the theory of rigid frameworks. In: Discrete Geometric Analysis,
{\it MSJ Memoirs}, vol. 34, pp. 33-112, 2016.


\bibitem{JKT}{\sc T. Jord\'an, Cs. Kir\'aly, and S. Tanigawa}, Generic global rigidity of body-hinge frameworks,
{\it J. Combin. Theory}, Series B 117, 59-76, 2016.

\bibitem{JW} {\sc T. Jord\'an and W. Whiteley,}
Global rigidity,  {\it in} J. E.
Goodman, J. O'Rourke, and C. D. T\'oth (eds.), {\it Handbook of Discrete and Computational
Geometry,} 3rd ed., CRC Press, Boca Raton, pp. 1661-1694, 2018.

\bibitem{JT} {\sc T. Jord\'an and S.Tanigawa}, Global rigidity of triangulations
with braces, {\it J. Comb. Theory} Ser. B., 136, pp. 249-288 (2019).





\bibitem{KM} {\scshape Cs. Kir\'aly and A. Mih\'alyk\'o},
Fast algorithms for sparsity matroids and the global rigidity augmentation problem, Egerv\'ary Research Group, Budapest,
TR-2022-05, 2022.


\bibitem{laman} {\scshape G. Laman},
\newblock On graphs and rigidity of plane
skeletal structures,
\newblock {\itshape J. Engineering Math.} 4 (1970),
331-340.







\bibitem{oxley} {\scshape J.G. Oxley},
\newblock {\itshape Matroid theory},
\newblock Oxford Science Publications.
The Clarendon Press, Oxford University Press, New York, 1992. xii+532 pp.



\bibitem{Saxe} {\scshape J.B. Saxe}, Embeddability of weighted graphs 
in $k$-space is strongly NP-hard, Technical report, Computer Science Department,
Carnegie-Mellon University, Pittsburgh, PA, 1979.


\bibitem{SW} {\sc B. Schulze and W. Whiteley},
Rigidity and scene analysis, {\it in} 
J.E. Goodman, J. O'Rourke, C.D. T\'oth (eds.),
 {\it Handbook of Discrete and Computational
Geometry,} 3rd ed., CRC Press, Boca Raton, 
2018.





\bibitem{Tani} {\scshape S. Tanigawa}, Sufficient conditions for the global rigidity of graphs,
{\it J. Combin. Theory}, Ser. B., Vol. 113, July 2015, Pages 123-140.








\end{thebibliography}
\end{document}